\documentclass[11pt]{amsart}

\usepackage{amsmath}
\usepackage{amsthm}
\usepackage{stmaryrd}
\usepackage{amsfonts,mathrsfs,amssymb}
\usepackage{times}
\usepackage{fullpage}
\title{$L^2$ Extension of $\dbar$-closed forms from a hypersurface}
\author{Jeffery D. McNeal$^{\dagger}$} 
\thanks{$\dagger$ Partially supported an NSF grant}
\email{mcneal@math.ohio-state.edu}
\address{
Department of Mathematics \newline \indent
Ohio State University \newline \indent
Columbus, OH 43210-1174}
\author{Dror Varolin$^{\ddagger}$} 
\thanks{$\ddagger$ Partially supported NSF grant DMS-1001896}
\email{dror@math.sunsysb.edu}
\address{
Department of Mathematics \newline \indent
Stony Brook University \newline \indent
Stony Brook, NY 11794-3651}

\newcommand{\noi}{\noindent}

\newcommand{\co}{{\mathcal O}}

\newcommand{\sC}{{\mathscr C}}
\newcommand{\sd}{{\mathscr D}}

\newcommand{\sh}{{\mathscr H}}

\newcommand{\ve}{\varepsilon}
 
\newcommand{\vp}{\varphi}

\newcommand{\C}{{\mathbb C}}
\newcommand{\D}{{\mathbb D}}

\newcommand{\R}{{\mathbb R}}

\newcommand{\red}{\hfill $\diamond$}

\newcommand{\di}{\partial}
\newcommand{\dbar}{\bar \partial}

\newcommand{\re}{{\rm Re\ }}

\newcommand{\relcomp}{\subset \subset}

\newcommand{\gradbar}{{\rm grad}^{''} \!\!}
\newcommand{\ii}{\sqrt{-1}}

\newcommand{\emb}{\hookrightarrow}
\newcommand{\tensor}{\otimes}

\def\XXint#1#2#3{{\setbox0=\hbox{$#1{#2#3}{\int}$} 
\vcenter{\hbox{$#2#3$}}\kern-.5\wd0}}

\usepackage{hyperref}

\begin{document}

\theoremstyle{plain}
\newtheorem{thm}{\sc Theorem}
\newtheorem*{s-thm}{\sc Theorem}
\newtheorem{lem}{\sc Lemma}[section]
\newtheorem{d-thm}[lem]{\sc Theorem}
\newtheorem{prop}[lem]{\sc Proposition}
\newtheorem{cor}[lem]{\sc Corollary}

\theoremstyle{definition}
\newtheorem{conj}[lem]{\sc Conjecture}
\newtheorem{prob}[lem]{\sc Open Problem}
\newtheorem{defn}[lem]{\sc Definition}
\newtheorem{qn}[lem]{\sc Question}
\newtheorem{ex}[lem]{\sc Example}
\newtheorem{rmk}[lem]{\sc Remark}
\newtheorem{rmks}[lem]{\sc Remarks}
\newtheorem*{ack}{\sc Acknowledgment}




\maketitle

\setcounter{tocdepth}1



\vskip .3in



\begin{abstract}
We establish $L^2$ extension theorems for $\bar \partial$-closed $(0,q)$-forms with values in a holomorphic line bundle with smooth Hermitian metric, from a smooth hypersurface on a Stein manifold.  Our result extends (and gives a new, perhaps more classical, proof of) a theorem of Berndtsson on compact K\"ahler manifolds, which itself is a sharpening of the theorem of Koziarz.  The proof makes use of the Kohn solution, which is the solution of an (interior) elliptic problem, to handle the well-known regularity issues.  As such, our methods require the line bundle to be equipped with a smooth metric.
\end{abstract}

\section{Introduction}
Let $X$ be a K\"ahler manifold of complex dimension $n$ with smooth K\"ahler metric $\omega$, and $Z \subset X$ a smooth complex hypersurface.  Let $L \to X$ be a holomorphic line bundle with a possibly singular Hermitian metric $e^{-\vp}$ whose singular locus does not lie in $Z$, i.e., such that $e^{-\vp}|_{Z}$ is a metric for $L|_Z$.  Assume also that the line bundle $E_Z \to X$ associated to the divisor $Z$ has a holomorphic section $f_Z$ such that $Z = \{ x \in X\ ;\ f_Z(x) = 0\}$, and a singular Hermitian metric $e^{-\lambda _Z}$, such that 
\[
\sup _X |f_Z|^2e^{-\lambda _Z} = 1.
\]
A basic problem in complex analytic geometry is whether all $L$-valued, $\dbar$-closed $(0,q)$-forms on $Z$ satisfying certain $L^2$ estimates can be extended to 
$\dbar$-closed forms on $X$,  with the extension satisfying $L^2$ estimates as well.  When $q=0$, the problem is one of $L^2$ extension of {\it holomorphic} sections of $L$ from $Z$ to $X$.  There are various conditions under which such extension can be carried out, e.g. \cite{dem-ot, mv-gain}; these results are often called {\it Extension Theorems of Ohsawa-Takegoshi Type} because of the pioneering work in \cite{ot-87}.

When $q \ge 1$, the problem requires clarification for the following reasons.

\begin{enumerate}
\item There are two natural choices for the restriction to $Z$ of an $L$-valued $(0,q)$-form on $X$.  
\begin{enumerate}
\item one can pull back the form via the natural inclusion $\iota :Z \emb X$ to produce an $L$-valued $(0,q)$-form on $Z$, which we call the {\it intrinsic restriction}, or 
\item one can restrict the points at which the $L$-valued $(0,q)$-form on $X$ is specified to lie in $Z$.  That is to say, the restriction is a section of the restricted vector bundle $(L\tensor \Lambda ^{0,q}_X)|_Z \to Z$.  We call such a form the {\it ambient restriction}.
\end{enumerate}
\item The class of data on $Z$ that can be extended to  $\dbar$-closed forms on $X$ depends on the notion of restriction used.
In the intrinsic case, a form on $Z$ is the restriction of a $\dbar$-closed form if and only if this form is itself $\dbar$-closed, but there is no natural $\dbar$ operator for the restricted vector bundle $(L\tensor \Lambda ^{0,q}_X)|_Z \to Z$.  
This point is elaborated in Section \ref{S:3} (c.f. Definition \ref{restriction-defns}), where criteria for the compatibility of ambient restriction and $\dbar$-closedness are given.  
\item $\dbar$-closed forms are not automatically smooth, since $\dbar$ is not elliptic unless $q=0$. Also, it is clear that some regularity of forms on $X$ is required to make sense of their restriction to $Z$. In this paper, we consider the problem where the forms to be extended are smooth, and we seek smooth extensions.  Furthermore, we assume that our data is smooth; in particular the metrics $e^{-\vp}$ and $e^{-\lambda _Z}$ are smooth.  The extension problem makes sense for continuous forms and singular metrics, as well as more general settings, but handling these cases requires additional technical considerations.
\end{enumerate}
Our goal is to establish several extension results, beginning with the following theorem.

\begin{thm}[Ambient $L^2$ extension]\label{ambient-main}
Let the notation be as above, and denote by $\iota :Z \emb X$ the natural inclusion.  Assume that 
\[
\ii \left ((\di \dbar (\vp - \lambda_Z) + {\rm Ricci}(\omega))\right ) \wedge \omega ^{q}\ge 0
\]
and
\[
\ii (\di \dbar (\vp - (1+\delta)\lambda_Z) + {\rm Ricci}(\omega)) \wedge \omega ^{q} \ge 0
\]
for some constant $\delta > 0$.  Then there is a constant $C>0$ such that for any smooth section $\xi$ of the vector bundle $(L\tensor \Lambda ^{0,q}_X)|_Z \to Z$ satisfying
\[
\dbar (\iota ^* \xi )= 0 \quad \text{and} \quad \int _Z\frac{ |\xi|_{\omega}^2 e^{-\vp}}{|df_Z|^2e^{-\lambda _Z}} \omega^{n-1} <+\infty ,
\]
there exists a smooth $\dbar$-closed $L$-valued $(0,q)$-form $u$ on $X$ such that
\[
u|_Z = \xi \quad \text{and} \quad \int _X|u|_{\omega}^2 e^{-\vp} \frac{\omega^n}{n!} \le \frac{C}{\delta}\int _Z\frac{ |\xi|_{\omega}^2 e^{-\vp}}{|df_Z|^2e^{-\lambda _Z}} \frac{\omega^{n-1}}{(n-1)!} .
\]
The constant $C$ is universal, i.e., it is independent of all the data.
\end{thm}

For a smooth section $\xi$ of $L\tensor \Lambda ^{0,q}_X)|_{Z}$, the pullback $\iota ^* \xi$ is a well-defined $L$-valued $(0,q)$-form on $Z$ (c.f. Equation \eqref{intrinsic-part-defn} in Paragraph \ref{second-amb-par}).  Note that if $\eta$ is an $L$-valued $(0,q)$-form on $Z$, then the orthogonal projection $P: T^{0,1}_X|_Z \to T^{0,1}_Z$ induced by the K\"ahler metric $\omega$ maps $\eta$ to an ambient $L$-valued $(0,q)$-form, i.e., a section $P^*\eta$ of $(L \tensor \Lambda ^{0,q}_X)|_Z \to Z$, by the formula
\begin{equation}\label{kop}
\left < P^*\eta , \bar v_1 \wedge ... \wedge \bar v_q\right > := \left < \eta ,  (P\bar v_1) \wedge ... \wedge (P\bar v_q) \right> \quad \text{in }L_z
\end{equation}
for all $v_1,...,v_q \in T^{*0,1}_{X,z}$.  The map $P^*$ is an isometry for the pointwise norm on $(0,q)$-forms induced by $\omega$, and since $\iota ^*P^*\eta = \eta$, the hypotheses of Theorem \ref{ambient-main} apply to $\xi = P^*\eta$, and we obtain the following theorem.

\begin{thm}[Intrinsic $L^2$ extension] \label{intrinsic-main}
Suppose the hypotheses of Theorem \ref{ambient-main} are satisfied.  Then there is a universal constant $C>0$ such that for any smooth $\dbar$-closed $L$-valued $(0,q)$-form $\eta$ on $Z$ satisfying 
\[
\int _Z\frac{ |\eta|_{\omega}^2 e^{-\vp}}{|df_Z|^2e^{-\lambda _Z}} \omega^{n-1} <+\infty ,
\]
there exists a smooth $\dbar$-closed $L$-valued $(0,q)$-form $u$ on $X$ such that, with $\iota :Z \emb X$ denoting the natural inclusion,
\[
\iota ^* u = \eta \quad \text{and} \quad \int _X|u|_{\omega}^2 e^{-\vp} \frac{\omega^n}{n!} \le \frac{C}{\delta}\int _Z\frac{ |\eta|_{\omega}^2 e^{-\vp}}{|df_Z|^2e^{-\lambda _Z}} \frac{\omega^{n-1}}{(n-1)!} .
\]
\end{thm}

An analog of Theorem \ref{intrinsic-main} in the situation where $X$ is a compact complex manifold was recently proved by Berndtsson in \cite{bo-forms}.  Berndtsson's work employs the method of solving the $\dbar$-equation for a current, developed in \cite{bs} .  Our approach is rather different.

In \cite{bo-forms} an interesting improvement of Theorem \ref{intrinsic-main} was observed in the compact case: if $X$ is a compact K\"ahler manifold and $Z$ is a smooth hypersurface,  a $\dbar$-closed $L$-valued form $\beta$ on $Z$ is $\dbar$-exact if and only if there are $L^2$-extensions of $\beta$ to $X$ having arbitrarily small $L^2$ norm.  Equivalently, the estimate for the extension depends only on the cohomology class of the smooth form to be extended.  

An analogous improvement is possible in the Stein case.  To state the result, we begin with the following definition.

\begin{defn}[Cohomological energy]
Let $\eta$ be a $\dbar$-closed $L|_Z$-valued $(0,q)$-form.  The number 
\[
\kappa (\eta) := \inf _{\theta}\int _Z \frac{|\theta|^2e^{-\vp}}{|df_Z|^2e^{-\lambda _Z}} \frac{\omega^{n-1}}{(n-1)!},
\]
where the infimum is taken over all forms $\theta$ such that $\theta = \eta + \dbar u$ for some $L$-valued $(0,q-1)$-form $u$ satisfying  
\[
\int _Z \frac{|u|^2e^{-\vp}}{|df_Z|^2e^{-\lambda _Z}}\omega ^{n-1} < +\infty,
\]
is called the {\it Cohomological energy} of $\eta$ ($C$-energy, for short).
\red
\end{defn}

\begin{rmk}
Since $Z$ is Stein, every $\dbar$-closed form with values in a holomorphic line bundle is $\dbar$-exact.  However, it may be the case that if $\eta = \dbar u$ then necessarily 
\[
\int _Z \frac{|u|^2e^{-\vp}}{|df_Z|^2e^{-\lambda _Z}} \omega ^{n-1} = +\infty.
\]
Thus the weighted $L^2$-cohomology need not agree with the smooth cohomology (which in the present setting is always trivial).  For example, consider the case $Z= \C$ with coordinate $z$ and the Euclidean K\"ahler form.  Let $\vp(z) = 2 \log (1+|z|^2)$ and let $\eta = d\bar z$.  Then 
\[
\int _{\C} \frac{|\eta|^2dA(z)}{(1+|z|^2)^2} < +\infty.  
\]
On the other hand, for any $u$ such that $\dbar u = \eta$ there exists $v \in\co (\C)$ such that $u = \bar z + v$, and since our weight is radially symmetric, $v \perp \bar z$ and we have 
\[
\int _{\C} \frac{|u|^2dA(z)}{(1+|z|^2)^2} = \int _{\C} \frac{|z|^2dA(z)}{(1+|z|^2)^2}  + \int _{\C} \frac{|v|^2dA(z)}{(1+|z|^2)^2} \ge \int _{\C} \frac{|z|^2dA(z)}{(1+|z|^2)^2},
\]
which is infinite.
\red
\end{rmk}

\begin{rmk}
The $C$-energy of a cohomology class is realized by an actual current, though this current may well have non-smooth coefficients.  Indeed, one takes the $\dbar$-closed form $\theta$ in the $\dbar$-cohomology class of $\eta$ that is orthogonal to the kernel of $\dbar$.  Since it is a minimizer, this form exists by the completeness of Hilbert spaces, and is unique by an elementary argument.  Establishing regularity of this minimizer is usually not elementary.
\red
\end{rmk}

Our third theorem is the following

\begin{thm}[Intrinsic extension with cohomology bounds]\label{main-i-coh}
Assume that the hypotheses of Theorem \ref{ambient-main} hold.  Let $\ve > 0$ be given.  Then for any $\dbar$-closed, $L|_Z$-valued $(0,q)$-form $\eta$ satisfying
\[
\int _Z \frac{|\eta|^2e^{-\vp}}{|df_Z|^2e^{-\lambda _Z}} \omega^{n-1} <+\infty ,
\]
there exists a $\dbar$-closed, $L$-valued $(0,q)$-form $u_{\ve}$ on $X$ such that if $\iota : Z \emb X$ is the natural inclusion then
\[
\iota ^* u_{\varepsilon} = \eta \quad \text{and} \quad \int _X |u_{\ve}|^2e^{-\vp} \frac{\omega^n}{n!} \le C (\kappa (\eta) +\ve),
\]
where the constant $C$ depends only on $\delta$.
\end{thm}

The analog of Theorem \ref{main-i-coh} for ambient extension constitutes our final result.

\begin{thm}[Ambient extension with cohomology bounds]\label{main-a-coh}
Assume that the hypotheses of Theorem \ref{ambient-main} hold.  Let $\ve > 0$ be given.  Then for any smooth section $\xi$ of $(L\tensor \Lambda ^{0,q}_X)|_Z \to Z$ satisfying  
\[
\dbar \iota ^*\xi = 0 \quad \text{and} \quad \int _Z \frac{|\xi|^2e^{-\vp}}{|df_Z|^2e^{-\lambda _Z}} \omega^{n-1} <+\infty ,
\]
there exists a $\dbar$-closed, $L$-valued $(0,q)$-form $u_{\ve}$ on $X$ such that if $\iota : Z \emb X$ is the natural inclusion then
\[
u_{\varepsilon}|_Z = \xi \quad \text{and} \quad \int _X |u_{\ve}|^2e^{-\vp} \frac{\omega^n}{n!} \le C (\kappa (\iota ^*\xi) +\ve),
\]
where the constant $C$ depends only on $\delta$.
\end{thm}

\begin{rmk}
In Theorems \ref{main-i-coh} and \ref{main-a-coh}, if the forms to be extended already realize the $C$-energy of their $L^2$-cohomology class
(i.e. have minimal norm in their class) , then the error $\ve$ in the estimate can be removed; indeed, in this case, Theorems \ref{intrinsic-main} and \ref{ambient-main} respectively apply.
\red
\end{rmk}

As with Theorems \ref{ambient-main} and \ref{intrinsic-main}, Theorem \ref{main-a-coh} on ambient extensions generalizes Theorem \ref{main-i-coh} on intrinsic extensions.  However, the proofs of the various theorems are interconnected, and we have stated them in the order that reflects the logic of our proofs.

The results above immediately imply analogous extension theorems in projective manifolds, since by Lemme 6.9 in \cite{dem-82a}) a solution of $\dbar$ with $L^2$ estimates on the complement of a divisor extends across that divisor.  Thus, the projective case is reduced to the Stein case upon removing a very ample hypersurface.  Similarly, our results hold for so-called {\it essentially Stein manifolds}, i.e., manifolds that are Stein after one removes a hypersurface.

The problem of extending $\dbar$-closed $(0,q)$-forms with values in a holomorphic line bundle seems to have first been considered by Manivel \cite{manivel}, who indicated that his proof for the case $q=0$ extended to the case of higher $q$.  That this was not so was indicated by Demailly in the paper \cite{dem-ot} which also contained a conjectural approach to a proof.  Demailly's conjectural approach has still not been realized, and it would be very interesting to establish, since it might, among other things, lead to establishing the extension theorem in the setting where the Hermitian metrics are singular.  The first key breakthrough was made by Koziarz \cite{koz}, who established the extension of cohomology classes on compact manifolds, but with estimates that depend on the underlying compact manifold, and methods that do not extend to any non-compact setting.  By observing that exact forms extend on compact manifolds to closed forms with arbitrarily small $L^2$ norm, Berndtsson \cite{bo-forms} showed that Koziarz's theorem essentially solved the extension problem on compact manifolds,  and moreover Berndtsson found the improved estimates that one hoped for in the compact case.  There is a link between Berndtsson's approach and Demailly's program; however, Berndtsson established his results for compact manifolds. There are some difficulties in passing to the open case, though it is conceivable that Berndtsson's approach could be extended to open K\"ahler manifolds, perhaps with some additional convexity assumptions.

\begin{ack}
The second author would like to thank Seb Boucksom for useful discussions, and Bo Berndtsson for sending him an early version of \cite{bo-forms}.  Part of this work was done while the second author was visiting the University of Michigan in the Winter semester of 2012, and he thanks the mathematics department for excellent working conditions, and especially Mattias Jonsson for many interesting discussions and for being a wonderful host.
\end{ack}

\section{Notation and Background}

In this section, some background material needed in the rest of the paper is reviewed.  Everything discussed here is well-known, but summarized for readers who are better versed in either the geometric or the analytic aspects of the area, though perhaps not both.  The expert may rapidly skim or skip this section.

\subsection{Review of geometry of Hermitian vector bundles}

A (complex linear) connection $D$ for a vector bundle $E \to X$ is a differential operator of (not necessarily pure) order $1$ that maps sections of $E$ to $1$-forms with values in $E$, i.e., sections of $(\C \tensor T^*_X) \tensor E$, that also satisfies the Leibniz rule
\[
D(fs) = df \tensor s +fDs.
\]
This Leibniz rule therefore determines the connection from its values on a frame, and in such a frame for $E \to X$ (which identifies $E$ with a trivial vector bundle) the connection differs from the derivative of the section by a term of order $0$; we say that it is a {\it twisted exterior derivative}.  The connection induces {\it twisted exterior derivatives} on $E$-valued differential forms, i.e., it is extended to a map sending $E$-valued differential forms of order $r$ to those of order $r+1$:
\[
D : \Gamma (X, \sC _X^{\infty}(\Lambda ^r (\C \tensor T^*_X) \tensor E)) \to \Gamma (X, \sC_X ^{\infty}(\Lambda ^{r+1} (\C \tensor T^*_X) \tensor E)).
\]
Again the connection is required to satisfy a Leibniz rule compatible with the skew symmetry of differential forms: if we have a local section $s$ of $E$ and a differential $r$-form $\beta$, then the $E$-valued differential $r$-form $\beta \tensor s$ has covariant derivative 
\begin{equation}\label{ext-connection}
D (\beta \tensor s) = d\beta \tensor s + (-1)^{r} \beta \wedge Ds.
\end{equation}
As usual, one multiplies by $(-1)^r$ to pick out the skew-symmetric part of the second derivatives.

In general there are many connections for a given vector bundle, but when the vector bundle is holomorphic and equipped with a Hermitian metric,  there is exactly one connection $D$, the so-called {\it Chern connection}, that is compatible with the metric, i.e., satisfies
\[
d(s,\sigma) = (Ds,\sigma) + (s, D\sigma),
\]  
and splits into its $(1,0)$ and $(0,1)$ parts as  
\[
D = D^{1,0} + \dbar.
\]
It is not difficult to show that for a holomorphic line bundle $E \to X$ with Hermitian metric $h$, the Chern connection is given in terms of a frame $\xi$ by 
\[
D (f\xi) = (d f -(-1)^r f \wedge \di \vp ^{(\xi)})\tensor \xi,
\]
where $f$ is a $(0,r)$-form and $\vp^{(\xi)} = - \log h(\xi,\bar \xi)$.  A simple calculation shows that the curvature $\Theta$ of the Chern connection $D$, defined to be the operator $D^2$, is given by the formula 
\[
D^2(f\tensor \xi) = \di \dbar \vp ^{(\xi)} \wedge f \tensor \xi.
\]
As the reader can easily check, the form $\di \dbar \vp ^{(\xi)}$ is globally defined and independent of the frame.  This invariance accounts for our (common) abusive notation $h = e^{-\vp}$ and $\di \dbar \vp$ for the metric and curvature respectively.  

In the rest of the paper, our $(0,r)$-forms will take values in a line bundle.

\subsection{Pointwise inner products and Hermitian forms for $E$-valued $(0,r)$-forms}

Equipped with the K\"ahler metric $\omega$ for $X$ and the line bundle metric $e^{-\vp}$ for $E \to X$, we can define $L^2$ spaces of $E$-valued $(0,r)$-forms as follows.  Let 
\[
dV_{\omega} := \frac{\omega ^n}{n!}
\]
be the usual volume form associated to the K\"ahler metric $\omega$.  Suppose $\beta^1,\beta^2$ are $E$-valued $(0,r)$-forms, given locally with respect to a frame $\xi$ of $E$ by $\beta^i = f^i \tensor \xi$.  Then we have a well-defined $(r,r)$-form 
\[
\beta^1 \wedge \overline{\beta^2} e^{-\vp} = f^1 \wedge \overline{f^2} e^{-\vp ^{(\xi)}},
\]
and therefore a well-defined function $\left < \beta^1 , \beta ^2\right >_{\omega}e^{-\vp}$ defined by the equality of the $(n,n)$-forms 
\[
\left < \beta^1 , \beta ^2\right >_{\omega}e^{-\vp} dV_{\omega}  = \frac{(\ii)^{r(r+2)}}{r!(n-r)!} \omega ^{n-r} \wedge \beta^1 \wedge \overline{\beta^2} e^{-\vp}.
\]
We shall write 
\[
|\beta|^2_{\omega} e^{-\vp} = \left < \beta , \beta \right >_{\omega}e^{-\vp}.
\]
Finally, given a Hermitian $(1,1)$-form $\Theta$, we define a Hermitian form 
\[
\left < \Theta\beta ^1,\beta^2 \right >_{\omega}e^{-\vp}.
\]
on $E$-valued $(0,r)$-forms induced from $\Theta$, $e^{-\vp}$, and $\omega$.  In this formula, for an $E$-valued $(0,r)$-form $\beta$, the form $\Theta \beta$ is an $E$-valued $(0,r)$-form defined as follows.  The K\"ahler form induces a duality between $(1,1)$-forms and $(0,1)$-forms with values in $T^{0,1}_X$.  We can then contract any $(0,r)$-form with $\Theta$ to obtain a new $(0,r)$-form.  In terms of an orthonormal frame of $(1,0)$ forms $\alpha ^1,...,\alpha ^n$ for $\omega$ we can write $\Theta = \Theta _{i\bar j} \alpha ^i \wedge \bar \alpha _j$, and then $\Theta$ acts on the $(0,r)$-form $\beta = \beta _{\bar I} \bar \alpha ^I$ by 
\[
\Theta \beta = \Theta _{i \bar j_k} \beta _{\bar j_1 ... \bar j_{k-1} (\bar i)_k \bar j_{k+1} \bar j_r} \bar \alpha ^J.
\]
It is not difficult to show that this form is positive on twisted $(0,r)$-forms if and only if 
\[
\Theta \wedge \omega ^{r-1}
\]
is a positive $(r,r)$-form.

\subsection{Positivity of metrics and convexity of boundaries for $(0,q)$-forms}

In our main theorems, non-negativity of the $(q,q)$-forms
\[
\ii (\di \dbar \vp +{\rm Ricci}(\omega)) \wedge \omega ^{q} \quad \text{and} \quad\ii (\di \dbar \vp +{\rm Ricci}(\omega) - \delta \di \dbar \lambda _Z) \wedge \omega ^{q}.
\]
is stipulated. 
Although a precise idea of what this non-negativity requires of the $(1,1)$-forms $\di \dbar \vp +{\rm Ricci}(\omega)$ and $\di \dbar \vp +{\rm Ricci}(\omega) - \delta\di \dbar \lambda _Z$ is not needed below, it is instructive to unravel these conditions somewhat.  To this end, let us examine more carefully the action of a Hermitian $(1,1)$-form $\Theta$ on an $L$-valued $(0,r)$-form $\beta$ at a point.  Choose a local orthonormal frame of $(1,0)$-forms $\alpha ^1, ..., \alpha ^n$ for $\omega$ that diagonalizes $\Theta$, i.e., such that
\[
\omega = \frac{\ii}{2} \left (\alpha ^1 \wedge \bar \alpha ^1 + ... + \alpha ^n \wedge \bar \alpha ^n\right ) \quad \text{and} \quad \Theta = \frac{\ii}{2} \left (\theta _1 \alpha ^1 \wedge \bar \alpha ^1 + ... + \theta _n \alpha ^n \wedge \bar \alpha ^n \right ).
\]
If $\xi$ is a frame for $L$, then we can write 
\[
\beta = \sum _{|J|=r} f_J \bar \alpha ^J \tensor \xi.
\]
We therefore have 
\[
\left < \Theta \beta ,\beta \right > _{\omega} e^{-\vp} dV_{\omega} =  \sum _{|J|=r} \left (\sum _{i \in J} \theta _i\right ) |f_J|^2e^{-\vp^{(\xi)}}  dV_{\omega}.
\]

\begin{defn}Let $(X,\omega)$ be a K\"ahler manifold.
\begin{enumerate}
\item[(P1)] A Hermitian $(1,1)$-form $\Theta$ is said to be $r$-positive with respect to $\omega$ if for all $(0,r)$-forms $\beta$, 
\[
\left < \Theta \beta , \beta \right > _{\omega} \ge 0.
\]
Equivalently, the sum of the smallest $r$ eigenvalues of $\Theta$ with respect to $\omega$ is non-negative.
\item[(P2)] We say that a metric $e^{-\vp}$ is $r$-positively curved with respect to $\omega$ if the form $\ii \di \dbar \vp$ is $r$-positive with respect to $\omega$.
\item[(P3)] The smooth boundary $\di \Omega$ of a domain $\Omega = \{ \rho < 0\}$ (where as usual smoothly bounded domain means $d\rho \neq 0$ at any point of $\di \Omega$) is said to be $r$-pseudoconvex if the form $\ii \di \dbar \rho$ is $r$-positive with respect to $\omega$ on the complex cotangent space $T^{*1,0}_{\di \Omega} := T^*_{\Omega} \cap JT^*_{\Omega}$.
\red
\end{enumerate}
\end{defn}

\begin{rmk}
It follows immediately that $r$-positivity (resp. $r$-pseudoconvexity) implies $s$-positivity (resp. $s$-pseudoconvexity) for any $s > r$.
\red
\end{rmk}

\begin{rmk}
In the strongest case, when $r=1$, the usual notions of positivity and pseudoconvexity respectively are recovered.  Note however that our notion of $r$-pseudoconvexity is not the same as Andreotti-Grauert convexity. 
In particular, the following example shows that when $r \ge 2$, the condition of $r$-positivity depends on the metric $\omega$ :
\red
\end{rmk}

\begin{ex}
Working in $\C ^n= \C^{n-1} \times \C$ with coordinates $z=(z',\zeta)$, consider the two metrics 
\[
\omega _1 = \ii \di \dbar |z|^2 \quad \text{ and } \quad \omega _2 = 2\ve \ii \di \dbar |z'|^2 + \ii d\zeta \wedge d\bar \zeta.
\]
Then for any $\ve \in (0,1)$, the form 
\[
\Theta = \ve \ii \di \dbar |z'|^2 - \ii d\zeta \wedge d\bar \zeta
\]
is $2$-positive with respect to $\omega _2$ but not with respect to $\omega _1$.
\red
\end{ex}

\subsection{The Bochner-Kodaira-Morrey-Kohn Identity}

The basic tool in the subsequent $L^2$ arguments is an integral identity for smooth forms in the domain of the $\dbar$-Laplacian, here called the Bochner-Kodaira-Morrey-Kohn identity.

\begin{d-thm}\label{bkmk}
Let $(X,\omega)$ be a K\"ahler manifold of complex dimension $n$, and $L \to X$ a holomorphic line bundle with smooth Hermitian metric $e^{-\vp}$.  Let $\Omega \relcomp X$ be a domain with smooth boundary, and let $\rho :X \to \R$ be a smooth function such that $\Omega = \{ \rho < 0\}$ and  $|d\rho|_{\omega} \equiv 1$ on $\di \Omega$.  Then for any smooth $L$-valued $(0,q+1)$-form $\beta$ in the domain of $\dbar ^*$, one has the identity 
\begin{eqnarray}\label{bk-id-general}
\nonumber&& \int _{\Omega} |\dbar ^* \beta|^2 _{\omega} e^{-\vp} dV_{\omega} + \int _{\Omega} |\dbar \beta|^2 _{\omega} e^{-\vp}dV_{\omega} \\
&& \qquad = \int _{\Omega} |\overline{\nabla} \beta|^2 _{\omega} e^{-\vp}dV_{\omega} + \int _{\Omega} \left < \ii (\di \dbar \vp + {\rm Ricci}(\omega) ) \beta, \beta \right > _{\omega} e^{-\vp}dV_{\omega}\\
\nonumber &&\qquad \qquad + \int _{\di \Omega}  \left < \ii (\di \dbar \rho) \beta, \beta \right > _{\omega} e^{-\vp} dS _{\di \Omega}.
\end{eqnarray}
\end{d-thm}

The operator $\dbar^*$ is the adjoint of $\dbar$ relative to the inner product determined by $\omega$ and $e^{-\vp}$. A proof of Theorem \ref{bkmk} can be found, e.g., in \cite{varolin-pcmi}.

\subsection{The twisted Bochner-Kodaira-Morrey-Kohn Identity}

The twisted identity is obtained from \eqref{bk-id-general} by decomposing the metric $e^{-\vp}$ for $L$ as 
\[
e^{-\vp} = \tau e^{-\psi}
\]
for a smooth positive function $\tau$ and a smooth metric $e^{-\psi}$ for $L$.  One computes that 
\[
\dbar ^* _{\vp} \beta = \dbar ^* _{\psi} \beta - \frac{1}{\tau} (\gradbar \tau) \rfloor \beta.
\]
Here $\gradbar \tau$ is the {\it gradient} $(0,1)$-vector field induced from the $(1,0)$-form $\di \tau$ by $\omega$, i.e., the $(0,1)$-vector field characterized by   
\[
\langle \xi, \overline{\gradbar \tau} \rangle_{\omega} = \di \tau (\xi), \quad \xi \in T^{1,0} _X.
\]
Observe that
\[
|(\gradbar \tau) \rfloor \beta|^2_{\omega} e^{-\vp} = \left < \ii (\di \tau \wedge \dbar \tau) \beta , \beta \right > _{\omega}e^{-\vp}.
\]
Next the curvatures of $e^{-\vp}$ and $e^{-\psi}$ are linked by the identity
\[
\di \dbar \vp = \di \dbar (\psi - \log \tau) = \di \dbar \psi - \frac{\di \dbar \tau}{\tau} + \frac{\di \tau \wedge \dbar \tau}{\tau ^2}.
\]
Substitution into \eqref{bk-id-general} yields the so-called {\it twisted Bochner-Kodaira-Morrey-Kohn Identity} 
\begin{eqnarray}\label{twisted-bk-id}
\nonumber&& \int _{\Omega} \tau |\dbar ^*_{\psi} \beta|^2 _{\omega} e^{-\psi} dV_{\omega} + \int _{\Omega} \tau |\dbar \beta|^2 _{\omega} e^{-\psi}dV_{\omega} \\
&&  = \int _{\Omega} \left <\ii \{ \tau (\di \dbar \vp + {\rm Ricci}(\omega) )  - \di \dbar \tau \}\beta, \beta \right > _{\omega} e^{-\psi}dV_{\omega} + 2\re \int _{\Omega} \left < \dbar ^* _{\psi} \beta , \gradbar \tau \rfloor \beta \right > _{\omega} e^{-\psi} dV_{\Omega}\\
\nonumber && \qquad \qquad + \int _{\Omega} \tau |\overline{\nabla} \beta|^2 _{\omega} e^{-\psi}dV_{\omega}  + \int _{\di \Omega}  \left < \ii(\tau\di \dbar \rho) \beta, \beta \right > _{\omega} e^{-\psi} dS _{\di \Omega},
\end{eqnarray}
which holds for all smooth forms in the domain of $\dbar ^*_{\psi}$.

By applying to the second integral on the right hand side of \eqref{bk-id-general} the Cauchy-Schwarz Inequality, followed by the inequality $ab \le Aa^2+A^{-1}b^2$, one obtains
\[
 2\re \left < \dbar ^* _{\psi} \beta , \gradbar \tau \rfloor \beta \right > _{\omega} \le A |\dbar ^*_{\psi}\beta |^2_{\omega} + A^{-1} \left < (\di \tau \wedge \dbar \tau)\beta ,\beta\right > _{\omega}.
\] 
Thus, the following inequality holds:

\begin{lem}[Twisted Basic Estimate] \label{tbe}
Let $(X,\omega)$ be a Stein K\"ahler manifold and let $L \to X$ be a holomorphic line bundle with smooth Hermitian metric $e^{-\psi}$.  Let $A$ and $\tau$ be positive functions with $\tau$ smooth.  Fix a smoothly bounded domain $\Omega \relcomp X$ such that $\di \Omega$ is pseudoconvex.  Then for any smooth $L$-valued $(0,q+1)$-form $\beta$ in the domain of $\dbar ^*_{\psi}$ one has the estimate 
\begin{eqnarray}\label{tbest}
\nonumber && \int _{\Omega} (\tau +A)|\dbar ^*_{\psi} \beta |^2_{\omega} e^{-\psi} dV_{\omega} + \int _{\Omega} \tau|\dbar  \beta |^2_{\omega} e^{-\psi} dV_{\omega} \\
&& \qquad \ge \int _{\Omega} \left <\ii \{ \tau (\di \dbar \psi + {\rm Ricci}(\omega) )  - \di \dbar \tau - A^{-1}\di \tau \wedge \dbar \tau \}\beta, \beta \right > _{\omega} e^{-\psi}dV_{\omega}.
\end{eqnarray}
\end{lem}

For pseudoconvex domains in $\C^n$, Lemma \ref{tbe} was independently proved in \cite{b-96}, \cite{mc-96} and \cite{s-96}. For forms vanishing on $b\Omega$, the inequality was proved even earlier in \cite{ot-87}.

\subsection{Ellipticity of the twisted $\dbar$ Laplacian}

Let $\tau$ and $A$ be smooth functions with $\tau$ and $\tau +A$ positive.  Setting 
\[
T := \dbar \circ \sqrt{\tau +A} \quad \text{and} \quad S = \sqrt\tau \circ \dbar,
\]
where $T$ acts on $L$-valued $(0,q)$-forms and $S$ acts on $L$-valued $(0,q+1)$-forms, define the twisted $\dbar$-Laplacian 
\[
\Box := TT^*+S^*S,
\]
which is a second order operator mapping smooth $L$-valued $(0,q+1)$-forms to smooth $L$-valued $(0,q+1)$-forms.  Note that if $\tau = 1$ and $A=0$, the standard $\dbar$-Laplacian is recovered, which is well-known to be elliptic (recall we are assuming $e^{-\vp}$ is smooth). Denote the standard $\dbar$-Laplacian by $\Box _{0}$ below.

The twisted basic estimate \eqref{tbest} will be used to invert the operator $\Box$ under certain assumptions on $\tau$ and $A$, but we will also need to know something about the regularity of $\Box$.  The regularity needed follows from elliptic theory and the following proposition.

\begin{prop}\label{box-is-elliptic}
If $\tau$ and $A$ are smooth, and if $\tau$ and $\tau +A$ are positive, then the operator $\Box$ is second order (interior) elliptic with smooth coefficients.
\end{prop}

\begin{proof}
A simple calculation shows that 
\[
\sd := TT^*+S^*S - \left ( (\tau+A) \dbar \dbar ^* + \tau \dbar ^* \dbar \right )
\]
is a differential operator (for $L$-valued $(0,q)$-forms) or order $1$.  Since 
\[
(\tau+A) \dbar \dbar ^* + \tau \dbar ^* \dbar
\]
is bounded below by the product of a positive function and the operator 
\[
\Box _{0} := \dbar \dbar ^* +\dbar ^* \dbar,
\]
the proof follows from the ellipticity of $\Box _{0}$.
\end{proof}

In our applications, $\tau$ and $A$ will both be strictly positive.  

\section{The extension problem without estimates}\label{S:3}

\subsection{Two notions of restriction}\label{rest-par}  The following definition was essentially already made in the introduction.

\begin{defn}\label{restriction-defns}
Let $\iota :Z \emb X$ be the natural inclusion of $Z$ in $X$.
\begin{enumerate}
\item[(3.1.a)] Say that an $L|_Z$-valued $(0,q)$-form $\eta$ on $Z$ is the {\it intrinsic restriction} of an $L$-valued $(0,q)$-form $\theta$ if 
\[
\iota ^*\theta = \eta.
\]
\item[(3.1.b)] Say that a section $\xi$ of the vector bundle $\left . \left ( L \tensor \Lambda ^{0,q}_X \right )\right |_Z$ is the {\it ambient restriction} of an $L$-valued $(0,q)$-form $\theta$ on $X$ if  
\[
\theta (z) = \xi(z)
\]
for all $z \in Z$.
\end{enumerate}
\end{defn}

\subsection{The extensions of $\dbar$-closed intrinsic  and ambient restrictions}\label{second-amb-par}

We begin with the statement of the characterization of intrinsic restrictions of $\dbar$-closed forms.  Note that if an $L$-valued $(0,q)$-form $\eta$ on $Z$ is of the form $\eta = \iota ^* \theta$ for some $\dbar$-closed, $L$-valued $(0,q)$-form $\theta$ on $X$, then 
\[
\dbar \eta = \dbar \iota ^*\theta = \iota ^* \dbar \theta = 0.
\]
The next proposition states that on a Stein manifold, the necessity of $\dbar$-closedness is also sufficient.

\begin{prop}\label{int-ext-no-est}
Let $\eta$ be a smooth $\dbar$-closed, $L$-valued $(0,q)$-form on $Z$.  Then there exists a smooth $\dbar$-closed, $L$-valued $(0,q)$-form $\tilde u$ on $X$ whose intrinsic restriction to $Z$ is $\eta$, i.e., such that 
\[
\iota ^*\tilde u = \eta,
\]
where $\iota :Z \emb X$ denotes the natural inclusion.  \end{prop}

\noi Proposition \ref{int-ext-no-est} will be proved below.

We turn next to the characterization of ambient restriction of $\dbar$-closed forms.  By contrast with the case of intrinsic restriction, the symbol $\dbar \xi$ is meaningless; the bundle $\left . \left ( L \tensor \Lambda ^{0,q}_X \right )\right |_Z \to Z$ does not admit a naturally defined notion of $\dbar$.  However, for a section $\xi$ of the latter bundle, it makes sense to define $\iota ^* \xi$ as an $L|_Z$-valued $(0,q)$-form on $Z$ by the formula
\begin{equation}\label{intrinsic-part-defn}
\iota ^*\xi (v_1,...,v_q) := \xi (d\iota (p) v_1,...,d\iota (p)v_q), \quad v_i \in T^{0,1}_{Z,p}.
\end{equation}
Moreover, if $\xi$ is the ambient restriction of some $\dbar$-closed $L$-valued $(0,q)$-form $\theta$, then 
\[
\dbar \iota ^*\xi = \dbar \iota ^* \theta = \iota ^* \dbar \theta = 0.
\]
Conversely, the following Proposition holds:

\begin{prop}\label{ext-no-estimates}
Let $\xi$ be a smooth section of the vector bundle $\left . (L \tensor \Lambda ^{0,q}_X)\right |_Z \to Z$, such that 
\[
\dbar \iota ^* \xi = 0.
\]
Then there is a smooth $\dbar$-closed $L$-valued $(0,q)$-form $\tilde u$ on $X$ such that $\tilde u |_Z = \xi$.
\end{prop}

There are several ways to prove Propositions \ref{int-ext-no-est} and \ref{ext-no-estimates}.  The proofs given below are modeled on
 the method used in the proofs of the main theorems, i.e., when we provide extensions with $L^2$ estimates.

\begin{rmk}
Note that if $q=0$, Propositions \ref{int-ext-no-est} and  \ref{ext-no-estimates} become identical, and say simply that any holomorphic function from a hypersurface in a Stein manifold has a holomorphic extension.  Thus the result is well-known for $q=0$, and in fact {\it hypersurface} can be replaced by {\it closed submanifold}, or even more generally by {\it Stein subvariety} (in which case we must define holomorphic functions as those functions that are holomorphic in a neighborhood).  Therefore we assume from here on that $q \ge 1$.
\red
\end{rmk}

\begin{lem}\label{3-weights}
Let $X$ be a Stein manifold with K\"ahler form and $H \to X$ a holomorphic line bundle.  Let $h$ be a locally bounded, $\dbar$-closed (in the sense of currents), $H$-valued $(0,r+1)$-form.  Then there exists a locally integrable $H$-valued $(0,r)$-form $u$ such that $\dbar u = h$.  Moreover, if $h$ is smooth on an open set $U \subset X$ then one can choose $u$ to be smooth on $U$ as well.
\end{lem}

\begin{rmk}
For the case $r=0$, the assertion follows from the elliptic regularity of $\dbar$, and thus in particular {\it every} solution of $\dbar u = h$ is smooth when this is the case for $h$.

However in higher bi-degree, $\dbar$ is not elliptic and so the deduction of smoothness is not automatic.  There is nevertheless a standard approach to the problem:  the {\it particular} solution to $\dbar u = h$ of minimal $L^2$-norm is (interior) elliptic.
\end{rmk}

\begin{proof}[Proof of Lemma \ref{3-weights}]
Fix a smooth metric $e^{-\vp}$ for $H$.  Let $\rho :X \to \R$ be a strictly plurisubharmonic exhaustion function such that $\di \dbar \rho$ grows sufficiently rapidly.  Let $\Omega_c := \{ \rho < c\}$ be a strictly pseudoconvex sublevel set of $\rho$, where $c>>0$.   We will work in the Hilbert space closures $L^2_{(0,s)} (dV_{\omega}, \vp +\rho)$ of the set of smooth $H$-valued $(0,s)$-forms on $\Omega _c$, $s=r-1, r,r+1$, with respect to the inner product  
\begin{equation}\label{inner-prods-rho}
\left ( f, g\right ) := \int _{\Omega _c} \left < f, g \right >_{\omega} e^{-(\vp+\rho)} dV_{\omega}.
\end{equation}
Denote by $\dbar_r$ and $\dbar_{r+1}$ the Hilbert space extension of the $\dbar$ operator acting on $L$-valued $(0,r)$-forms and $(0,r+1)$-forms on $\Omega _c$, respectively, and by $\dbar_r^*$ and $\dbar_{r+1}^*$ the formal adjoint of $\dbar_r$ and $\dbar_{r+1}$ with respect to the relevant inner product \eqref{inner-prods-rho}.  The weight $ e^{-(\vp + \rho)}$ is smooth, so by standard facts the smooth forms satisfying the $\dbar$-Neumann boundary condition are dense in the domain of the Hilbert space adjoint of $\dbar$, in the so-called graph norm.

Since $\di \dbar \rho > 0$, the Bochner-Kodaira identity \eqref{bk-id-general} implies that for any smooth $L$-valued $(0,r+1)$-form $f$ in the domain of $\dbar_r^*$
\begin{eqnarray*}
&& \int _{\Omega_c} |\dbar_r^*f|^2_{\omega} e^{-(\vp +\rho )} dV_{\omega} + \int _{\Omega} |\dbar_{r+1}f|^2_{\omega} e^{-(\vp +\rho)} dV_{\omega}  \\
&& \qquad \ge \int _{\Omega_c} \left < \ii(\di \dbar (\vp +\rho) +{\rm Ricci}(\omega))f, f\right >_{\omega}e^{-(\vp +\rho)} dV_{\omega} .
\end{eqnarray*}
Note that by taking $\di \dbar \rho$ sufficiently positive, we can ensure that 
\[
\left <  \ii (\di \dbar (\vp +\rho) +{\rm Ricci}(\omega)) f,f\right > _{\omega}\ge |f|^2_{\omega}.
\]
From this it follows that
\[
||\dbar_r^*f||^2 + ||\dbar_{r+1}f||^2 \ge ||f||^2,
\]
where the norms are given by \eqref{inner-prods-rho}.

Defining the operator
\[
\Box_{0} = \dbar_r\dbar_r^*+\dbar_{r+1}^*\dbar_{r+1},
\]
 the usual functional analysis argument gives a solution to the equation 
\[
\Box_{0} g_c = h \qquad\text{on }\Omega_c
\]
such that 
\begin{equation}\label{omega-c-est}
\int _{\Omega _c} \left|g_c\right|^2_{\omega} e^{-(\vp + \rho)} dV_{\omega} \le \int _{\Omega _c} |h|^2_{\omega} e^{-(\vp +\rho)} dV_{\omega}.
\end{equation}
Moreover, since $\dbar_{r+1} h=0$ by hypothesis, we have 
\[
0= \left(\dbar_{r+1}h, \dbar_{r+1}g_c\right) = \left(\dbar_{r+1}\dbar_r\dbar_r^*g_c+\dbar_{r+1}\dbar_{r+1}^*\dbar_{r+1}g_c, \dbar_{r+1}g_c\right)=\left\|\dbar_{r+1}^*\dbar_{r+1}g_c\right\|^2 
\]
and, hence, 
$$ 0=(\dbar_{r+1}^*\dbar_{r+1}g_c,g_c) = ||\dbar_{r+1}g_c||^2.$$
Therefore, setting $u_c = T^*g_c$, it follows that $Tu_c = h$.

Note that the estimate \eqref{omega-c-est} is uniform with respect to $c$.   We may therefore let $c \to \infty$ and extract via Alaoglou's Theorem a weak  solution to the equations 
\[
\Box_{0} g = h, \quad Sg = 0\qquad\text{on }\Omega.
\]

Turning to regularity, it is well-known that the operator $\Box_{0}$ is second-order (interior) elliptic.  Therefore $g$ lies in $H^2_{\ell oc}(X)$, the Sobolev space of measurable functions all of whose derivatives of order at most $2$ are $L^2_{\ell oc}$.  Moreover if $h$ is smooth then so is $g$. (Here we have no boundary conditions, we are working only with interior ellipticity and avoiding the significant complications of the $\dbar$-Neumann problem.)   

Finally, let $u := \dbar ^* g$.  The form $u$ is well-defined in the weak sense because $g$ is in $H^2_{\ell oc}(X)$.  Moreover, by construction, $\dbar u = \Box g = h$.  The proof is complete.
\end{proof}

\begin{proof}[Proof of Proposition \ref{int-ext-no-est}]
Locally the problem is trivial.  To see this, first choose local coordinates $(z_j,f_Z)$ where for each $j$, $z_j$ are coordinates in a unit ball $B_j \subset Z$, such that $\{ B_j\}_{j\ge 1}$ is a locally finite cover of $Z$. Then cover a neighborhood of $Z$ in $X$ by sets of the form $B_j \times D_j$ where $D_j$ is biholomorphic to the unit disk.  On each set $U_j := B_j \times D_j$, we may write 
\[
\eta _j := \eta |_{B_j} = \sum _{\# \alpha = q} f^j_{\bar \alpha} (z_j) d\bar z_j^{\alpha}.
\]
This form is well-defined on $U_j$ and is $\dbar$-closed, so gives the local $\dbar$-closed extension.  

These local extensions are now patched together to obtain a global $\dbar$-closed extension.  Observe that 
\[
\eta _i - \eta _j \equiv 0 \quad \text{to order $1$ on }Z\cap (U_i \cap U_j).
\]
Define 
\[
U_0 = X- \bigcup _{j \ge 1} (1-\ve_j)(\overline{B_j \times D_j}) \quad \text{and} \quad \eta _0 \equiv 0 \text{ on }U_0,
\]
where the numbers $\ve _j$ are chosen so small that $\{ (1-\ve _j)B_j\} _{j \ge 1}$ still form a cover of $Z$.

Fix a partition of unity $\{ \chi _j\}$ subordinate to $\{ U_j\}$ and write 
\[
\tilde \eta _j := \sum _k \chi _k (\eta _j - \eta _k) \quad \text{and} \quad  h _j := \sum _k \chi _k \frac{(\eta _j - \eta _k)}{f_Z}.
\]
Then 
\[
\tilde \eta _i - \tilde \eta _j = \sum _k \chi _k (\eta _i - \eta _k + \eta _k - \eta _j ) = (\eta _i - \eta _j) \sum _k \chi _k =\eta _i-\eta _j,
\]
and similarly 
\[
h_i - h_j = \frac{\eta _i-\eta _j}{f_Z}.
\]
Observe that all the $h_i$ are well-defined because $\eta _i - \eta _j$ vanishes identically to order $1$ on $Z\cap(U_i \cap U_j)$ for all $i$ and $j$.  Now,  
\[
\dbar \tilde \eta _i - \dbar \tilde \eta _j = \dbar \eta _i - \dbar \eta _j = 0-0=0
\]
and similarly 
\[
\dbar h _i - \dbar h_j = 0
\]
on $U_i \cap U_j$.  It follows that 
\[
\Theta := \dbar \tilde \eta _i \quad \text{and} \quad H := f_Z ^{-1}\dbar \tilde \eta _i \qquad \text{ on }U_i
\]
are globally well-defined, smooth $(0,q+1)$-forms with values in $L$ and $L-E_Z$ respectively.  By Lemma \ref{3-weights} there exists an $(L-E_Z)$-valued smooth $(0,q)$-form $v$ such that $\dbar v = H$.  Evidently the smooth $L$-valued $(0,q)$-form $u := f_Z v$ satisfies 
\[
\dbar u = f_Z H = \Theta
\]
and moreover $u$ vanishes at each point of $Z$.  Finally, define 
\[
\hat \eta _i := \tilde \eta _i - u.
\]
Then 
\[
\hat \eta _i - \hat \eta _j = \tilde \eta _i - \tilde \eta _j = \eta _i -\eta _j
\]
so that 
\[
\tilde \eta := \eta _i - \hat \eta _i \quad \text{on }U_i
\]
is a globally defined smooth $L$-valued $(0,q)$-form satisfying $\iota ^* \tilde \eta = \eta$ and $\dbar \tilde \eta = - \dbar \hat \eta _i =  \dbar u - \Theta = 0$ on any $U_i$.  The proof is thus finished.
\end{proof}

\begin{proof}[Proof of Proposition \ref{ext-no-estimates}]

Let $\xi$ be the section to be extended.  By hypothesis, $\dbar \iota ^* \xi = 0$.  Therefore Proposition \ref{int-ext-no-est} implies the existence of an $L$-valued $(0,q)$-form $\tilde \eta$ on $X$ such that $\dbar \tilde \eta = 0$ and $\iota ^*\tilde \eta = \iota ^* \xi$.  Define 
\[
\delta = \xi - \tilde \eta |_Z.
\]
Then $\iota ^* \delta = 0$, which means that 
\[
\delta = d\bar f_Z \wedge g
\]
for some kind of object $g$ that locally looks like an $L$-valued $(0,q-1)$-form on $Z$.  To see the global nature of $g$ (along $Z$), recall that, by the adjunction formula, $df_Z$, which is only well-defined on $Z$, can be thought of as a section of $N_Z^* \tensor (E_Z)|_Z$ where $N_Z^*$ is the co-normal bundle of $Z$.  Since $Z$ is smooth, $df_Z$ is nowhere zero on $Z$, so the latter line bundle is trivial.  Therefore $(E_Z^*)|_Z$ agrees with the conormal bundle of $Z$.  It follows that $g$ is an $L\tensor \overline{E_Z^*}$-valued $(0,q-1)$-form on $Z$.  Take any smooth extension of $g$ to an $L\tensor \overline{E_Z^*}$-valued $(0,q-1)$-form $\tilde g$ on $X$.  The form 
\[
\overline{f_Z} \tensor \tilde g
\]
is thus a globally defined $L$-valued $(0,q)$-form on $X$.  It follows that
\[
\tilde \delta := \dbar ( \overline{f_Z} \tensor \tilde g).
\]
is well-defined.

An easy calculation shows that, since $f_Z |_Z\equiv 0$,  
\[
\tilde \delta |_Z = \dbar (\overline{f_Z} \tensor \tilde g)|_Z = d\bar f_Z \wedge g.
\]
It follows that $\theta := \tilde \eta + \tilde \delta$ satisfies 
\[
\theta |_Z = \tilde \eta |_Z + \delta = \xi,
\]
and this is what was needed.  
\end{proof}

\section{Proof of Theorem \ref{ambient-main}}

Choose smoothly bounded pseudoconvex domains $\Omega _j$, $j =1,2,...$, such that 
\[
\Omega _j \relcomp \Omega _{j+1} \quad \text{and} \quad \lim _{j \to \infty} \Omega_j = \bigcup _{j \ge 1} \Omega _j = X.
\]

\subsection{An a priori estimate}

The first task is to obtain from the Twisted Basic Estimate \eqref{tbe} a suitable a priori estimate.  Begin by setting 
\[
e^{-\psi} := e^{-\vp+\lambda _Z}.
\]
Next we turn to the choices of the functions $A$ and $\tau$.   The choices made are similar to those made in \cite{dv-tak}, which in turn was based on the methods developed by us in \cite{mv-gain}.  Let 
\[
h(x):=2-x+\log (2e^{x-1}-1), \quad v := \log (|f_Z|^2e^{-\lambda _Z}) \quad \text{and} \quad a := \gamma - \delta \log (e^v +\ve ^2),
\] 
where $\delta > 0$ is as in the statement of the main theorems, $x \ge 1$, and $\gamma >1$ is a real number such that $a > 1$.  Define
\[
\tau := a+ h(a) \quad \text{and} \quad A := \frac{1+h'(a)^2}{-h''(a)}.
\]
As noted in \cite{dv-tak}, 
\begin{equation}\label{h'}
h'(x) = (2e^{x-1}-1)^{-1} \in (0,1) \quad \text{and} \quad h''(x) = \frac{-2e^{x-1}}{(2e^{x-1}-1)^2} < 0,
\end{equation}
and therefore 
\begin{equation}\label{tau-h'}
A = 2e^{a-1} \quad \text{and} \quad \tau \ge 1+h'(a).
\end{equation}
Moreover, the choices made guarantee that $-\di \dbar \tau - A^{-1} \di \tau \wedge \dbar \tau = (1+h'(a))(-\di \dbar a)$.  Finally, a straightforward computation yields
\begin{eqnarray*}
- \di \dbar a &=& \delta \di \dbar \log (e^v + \ve ^2) \\
&=& \frac{\delta e^v}{e^v + \ve ^2} \di \dbar v + \frac{4 \ve ^2  \di ( e^{v/2}) \wedge \dbar (e^{v/2})}{\mu ((e^{v/2})^2 + \ve ^2)^2}  \\
&=& -\delta \frac{e^v}{e^v + \ve ^2}  \di \dbar \lambda _Z +  \frac{4\delta \ve ^2 \di ( e^{v/2}) \wedge \dbar (e^{v/2})}{((e^{v/2})^2 + \ve ^2)^2}.
\end{eqnarray*}
In the last equality, the fact that 
\[
\ii \di \dbar v = \pi [Z]- \ii \di \dbar \lambda _Z, 
\]
where $[Z]$ is the current of integration over $Z$ is used.  Since $e^v |_Z \equiv 0$, the term involving the current of integration vanishes.

A direct computation now yields
\begin{eqnarray*}
&& \ii \left ( \tau (\di \dbar \psi + {\rm Ricci}(\omega)) - \di \dbar \tau - A^{-1} \di \tau \wedge \dbar \tau \right ) \wedge \omega ^{q} \\
&=& \ii \left ( \tau (\di \dbar (\vp-\lambda _Z) + {\rm Ricci}(\omega)) + (1+h'(a))(-\di \dbar a)\right )\wedge \omega ^{q} \\
&=& \left ( \tau -  (1+h'(a))\left (\frac{e^v}{e^v + \ve ^2}\right )\right ) \ii (\di \dbar (\vp - \lambda _Z) + {\rm Ricci}(\omega)) \wedge \omega ^{q}\\
&& + \ii (1+h'(a))\frac{e^v}{e^v + \ve ^2}\left ( (\di \dbar (\vp-\lambda _Z) + {\rm Ricci}(\omega)) - \delta \di \dbar \lambda _Z \right )\wedge \omega ^{q}\\
&& \qquad  +\ii (1+h'(a)) \left ( \frac{4\delta \ve ^2 \di(e^{v/2}\wedge \dbar (e^{v/2})}{((e^{v/2})^2 + \ve ^2)^2}\right )\wedge \omega ^{q}\\
&\ge&  \delta \left ( \frac{4 \ve ^2 \ii \di(e^{v/2}\wedge \dbar (e^{v/2})}{((e^{v/2})^2 + \ve ^2)^2}\right )\wedge \omega ^{q},
\end{eqnarray*}
where the last inequality, which is in the sense of Hermitian $(q+1,q+1)$-forms, follows from the hypotheses of Theorems \ref{intrinsic-main} and \ref{ambient-main}, as well as the properties \eqref{h'} and \eqref{tau-h'}.  Consequently, the following lemma has been proved:

\begin{lem}\label{a-priori}
Let $T:= \dbar \circ \sqrt{\tau +A}$ and $S= \sqrt{\tau} \dbar$.  Then under the hypotheses of Theorems \ref{ambient-main} and \ref{intrinsic-main} one has the estimate 
\[
||T^*\beta||^2 +||S\beta||^2 \ge \delta \int _{\Omega} 4 \frac{\ve ^2}{((e^{v/2})^2 +\ve ^2)^2} \left \langle \ii \{\di (e^{v/2})^2 \wedge \dbar (e^{v/2})^2\} \beta , \beta \right \rangle _{\omega} e^{-\psi} dV_{\omega}
\]
for any $L$-valued $(0,q+1)$-form $\beta$ in the domain of the adjoint $T^*$.
\end{lem}

\subsection{A solution of the twisted $\dbar$-Laplace equation with good estimates}

Let $\xi$ be a smooth section of the vector bundle $(L\tensor \Lambda ^q(T^{*0,1}_X))|_Z \to Z$ satisfying 
\[
\dbar \iota ^*\xi = 0 \quad \text{and} \quad \int _Z\frac{|\xi |^2_{\omega}e^{-\vp}}{|df_Z|^2e^{-\lambda_Z} }\omega ^{n-1} < +\infty.
\]
By Proposition \ref{ext-no-estimates} there is a smooth $\dbar$-closed $L$-valued $(0,q)$-form $\tilde \eta$ on $X$ such that $\tilde \eta |_Z= \xi$.  Since $\tilde \eta$ is smooth and $\Omega \relcomp X$, 
\[
\int _{\Omega} |\tilde \eta|^2_{\omega} e^{-\vp} dV_{\omega} < +\infty.
\]
Let $\nu > 0$ be a real number which we will eventually let tend to $0$.  Let $\chi \in \sC ^{\infty}_o([0,1])$ be a cutoff function with values in $[0,1]$ such that $\chi |_{[0,\nu]} \equiv 1$ and $|\chi '| \le 1+\nu$.  For $\ve >0$, set $\chi _{\ve} := \chi (\ve ^{-2} e^v)$ and define 
\[
\alpha _{\ve} := f_Z^{-1} \dbar \chi _{\ve} \tilde \eta,
\]
which is a $\dbar$-closed, $(L-E_Z)$-valued $(0,q+1)$-form on $X$. The goal is to solve the equation 
\[
(TT^*+S^*S) W_{\ve} = \alpha _{\ve} 
\]
with good $L^2$-estimates with respect to the weight $e^{-\vp +\lambda _Z}$, and good $\sC ^{\infty}_{\ell oc}$-estimates.  Toward this end, observe that for any smooth $L$-valued $(0,q+1)$-form $\beta$ in the domain of $T^*$ one has the estimate
\begin{eqnarray}\label{twist-box-est}
\nonumber |(\beta , \alpha _{\ve})|^2 &:=& \left | \int _{\Omega} \left < \beta , \alpha _{\ve}\right > e^{-\vp+\lambda _Z}dV_{\omega} \right |^2 \\
\nonumber &\le & \left ( \int _{\Omega}|\left < \beta , \alpha _{\ve}\right >_{\omega}|e^{-\vp+\lambda _Z} dV_{\omega} \right )^2\\
&=& \left ( \int _{\Omega} \left | \left < \beta , 2\ve ^{-2} \chi ' (\ve^{-2} e^v) \tilde \eta \wedge (|f_Z|^{-1}e^{v/2}) \dbar (e^{v/2}) \right >_{\omega} \right | e^{-\vp+\lambda _Z}dV_{\omega} \right )^2 \\
\nonumber &\le& \frac{1}{\delta} \left ( \int _{\Omega} \left | \ve ^{-2} \tilde \eta\chi '(\ve ^{-2}e^v) \right |^2_{\omega} \frac{(e^v+\ve ^2)^2}{\ve ^2} e^{-\vp} dV_{\omega} \right ) \int _{\Omega} \left | \left (\gradbar e^{v/2} \right )\rfloor \beta \right |^2_{\omega} \frac{\delta \ve ^2}{(e^v + \ve ^2)^2} e^{-\vp} dV_{\omega}\\
\nonumber  &\le & \frac{C_{\ve}}{\delta} \left (||T^*\beta||^2+||S\beta||^2 \right ) ,
\end{eqnarray}
where  
\[
C_{\ve} := \frac{4(1+\nu)^2}{\delta} \frac{1}{\ve ^2} \int _{\{e^v \le \ve ^2\}} |\tilde \eta|^2_{\omega} e^{-\vp} dV_{\omega} \  {\buildrel \ve \to 0 \over\longrightarrow} \ \frac{8\pi(1+\nu)^2}{\delta} \int _{\Omega \cap Z} \frac{|\xi|^2_{\omega}e^{-\vp}}{|df_Z|^2_{\omega} e^{-\lambda _Z}} \frac{\omega ^{n-1}}{(n-1)!}.
\] 

In analogy with in the proof of Lemma \ref{3-weights}, write $\Box := TT^*+S^*S$.  We solve the equation
\[
\Box V_{\ve} = \alpha _{\ve}
\]
in the usual way, as follows.  Let $\sh$ denote the Hilbert space closure of the set of all smooth $L$-valued $(0,q+1)$-forms $\beta$ such that the norm $||\cdot ||_{\sh}$ associated to the inner product
\[
(\beta,\gamma)_{\sh} := (T^*\beta,T^*\gamma)+(S\beta, S\gamma)
\]
is finite.  (Lemma \ref{a-priori} shows that $(\cdot ,\cdot)_{\sh}$ is an inner product.)  Define the functional $\ell :\sh \to \C$ by 
\[
\ell (\beta) := (\beta , \alpha _{\ve}) = \int _{\Omega} \left < \beta,\alpha _{\ve}\right >_{\omega} e^{-\vp+\lambda _Z} dV_{\omega}.
\]
Then the estimate \eqref{twist-box-est} shows that $\ell \in \sh ^*$, the dual space of bounded linear functionals on $\sh$, and the $\sh^*$-norm of $\ell$ is controlled by $\delta^{-1} C_{\ve}$.  By the Riesz Representation Theorem there exists $V_{\ve} \in \sh$ such that 
\begin{equation}\label{boxie-ests}
||V_{\ve} ||_{\sh}^2 = ||\ell||_{\sh^*}^2 \le \delta ^{-1}C_{\ve} \quad \text{and} \quad (V_{\ve},\beta) = (T^*\alpha _{\ve} , T^*\beta) + (S\alpha_{\ve}, S\beta).
\end{equation}
The latter says that $\Box V_{\ve} = \alpha _{\ve}$.  Moreover, since $\Box$ is elliptic, $V_{\ve}$ is smooth on $\Omega$.  Now, since $ST=0$ and $S \alpha _{\ve} = \sqrt{\tau} \dbar \alpha _{\ve} = 0$, we find that 
\[
0=(S\Box V_{\ve} , SV_{\ve}) = ||S^*SV_{\ve}||^2 
\]
and thus 
\[
||SV_{\ve}||^2 = (S^*SV_{\ve},V_{\ve}) = 0.
\]

Now let $\Omega =\Omega _j$ and obtain a smooth, $L-E_Z$-valued $(0,q+1)$-form $V_{\ve,j}$ such that 
\[
\Box V_{\ve,j} = \alpha _{\ve} \quad \text{and} \quad ||V_{\ve,j}||^2_{\sh _j} \le \frac{C_{\ve}}{\delta}.
\]

Set $v_{\ve,j} := T^*V_{\ve,j}$.  It follows that
\[
Tv_{\ve} = \Box V_{\ve} = \alpha _{\ve}.
\]
and the following theorem is proved.

\begin{d-thm}\label{T-eq}
The equation $Tv_{\ve,j}= \alpha _{\ve}$ has a smooth solution $v_{\ve,j}$ satisfying the $L^2$-estimate
\[
\int _{\Omega} |v_{\ve,j}|^2_{\omega} e^{-\vp+\lambda _Z} dV_{\omega} \le \frac{C_{\ve}}{\delta}.
\]
\end{d-thm}

\subsection{Construction of a smooth extension on $\Omega _j$ with uniform bound}

Set 
\[
u _{\ve,j} := \chi _{\ve} \tilde \eta - \sqrt{\tau +A} v_{\ve,j} \tensor f_Z.
\]
Then 
\[
u_{\ve,j} |_Z = \xi \quad \text{and} \quad \dbar u_{\ve,j} = f_Z \tensor (\alpha _{\ve} - Tv_{\ve,j}) = 0.
\] 
Since $\chi _{\ve}$ is bounded and supported on a set whose measure tends to $0$ with $\ve$, there exists $\ve _j >0$ sufficiently small so that whenever $\ve \le \ve _j$, one has 
\[
\int _{\Omega _j}|u_{\ve,j}|^2_{\omega}e^{-\vp} dV_{\omega} \le (1+o(1)) \int _{\Omega _j} (\tau +A) |v_{\ve,j}|^2_{\omega} |f_Z|^2e^{-\vp} dV_{\omega} \le \int _{\Omega _j} (e^v(\tau +A)) |v_{\ve,j}|^2_{\omega} e^{-\vp+\lambda _Z} dV_{\omega}.
\]
Now, $e^v (\tau +A) = (e^{-\gamma/\delta} e^{-a/\delta}-\ve ^2) (a+h(a) +2e^{a-1})$ remains bounded as $\ve \to 0$, and in that limit it is at most $5e^{(\gamma -1)/\delta}$.  It follows that for some $\ve _j$ sufficiently small, the estimate
\[
\int _{\Omega_j}|u_{\ve,j}|^2_{\omega}e^{-\vp} dV_{\omega} \le (1+o(1)) \frac{5e^{(\gamma -1)/\delta}}{\delta} \int _{Z} \frac{|\xi|^2_{\omega} e^{-\vp}}{|df_Z|^2e^{-\lambda _Z}} \frac{\omega ^{n-1}}{(n-1)!} \le \frac{C}{\delta} \int _{Z} \frac{|\xi|^2_{\omega} e^{-\vp}}{|df_Z|^2e^{-\lambda _Z}} \frac{\omega ^{n-1}}{(n-1)!}
\]
holds, as soon as $0< \ve\le \ve _j$.  Thus for any such $\ve >0$, $u_{\ve,j}$ gives the desired extension in $\Omega_j$.  Let us write 
\[
u _j := u_{\ve _j, j}.
\]
To summarize, for each $j$ we have found a  smooth $L$-valued $(0,q)$-form $u_j$ on $\Omega _j$ such that 
\begin{equation}\label{bound-for-min-ext}
\dbar u_j = 0, \quad u_j|_{Z\cap\Omega _j}=\xi, \quad \text{and} \quad \int _{\Omega_j} |u_j|^2e^{-\vp}dV_{\omega} \le \frac{C}{\delta} \int _{Z} \frac{|\xi|^2_{\omega} e^{-\vp}}{|df_Z|^2e^{-\lambda _Z}} \frac{\omega ^{n-1}}{(n-1)!}.
\end{equation}
In particular, the right hand side is independent of $j$.

\begin{rmk}
Observe that one may take $\gamma = 1+\delta$ in the above construction , which makes the constant $C$ independent of $\delta$.
\red
\end{rmk}

\subsection{Minimizing the norm of a smooth extension}

Consider the affine subspace 
\[
\sh _j \subset L^2(\omega, e^{-\vp}) 
\]
obtained by taking the closure of all smooth $L$-valued $(0,q)$-forms $\tilde u$ satisfying 
\begin{equation}\label{E:H-ext}
\dbar\tilde u = 0 \quad \text{and} \quad \tilde u|_{Z\cap \Omega _j} = \xi.
\end{equation}
By \eqref{bound-for-min-ext} $\sh _j$ is not empty.  Let $U_j$ be the element of $\sh _j$ having minimal norm.  

\begin{lem}\label{L:1} $U_j$ is orthogonal to the subspace
\[
V :=\left\{\beta\,\, L\text{-valued}\,\, (0,q)\text{-forms}:\dbar\beta =0, \beta|_{Z\cap \Omega _j} =0\right\}
\]
in $L^2\left(\omega, e^{-\vp}\right)$.
\end{lem}

\begin{proof}
Suppose there exists $\beta_0\in V$ such that $\left(U_j,\beta_0\right)=c\neq 0$. Consider $\alpha :=\frac{c\beta_0}{\|\beta_0\|^2}\in V$ and set $\tilde U_j=U_j-\alpha$. It follows that $\tilde U_j$ satisfies \eqref{E:H-ext}, but $\left\|\tilde U_j\right\|^2=\left\| U_j\right\|^2 -\frac{|c|^2}{\|\beta_0\|^2}$. This contradicts the minimality of $\| U_j\|$.
\end{proof}

\begin{lem}\label{L:2}
$U_j$ belongs to the domain of $\dbar^*$ and $\dbar^* U_j=0$.
\end{lem}

\begin{proof}
Elements in Dom$\left(\dbar^*\right)$ are those $L^2$ $(0,q)$-forms $u$ satisfying
\[
\left|\left(u,\dbar v\right)\right|\leq c\|v\|\qquad\forall v\in\text{Dom}(\dbar)
\]
for some constant $c$.

Let $\chi\in C^\infty\left(\mathbb{R}\right)$, $\chi(x)=0$ is $|x| <\frac 12$, $\chi(x)=1$ if $|x| >1$. For $\epsilon >0$, set $\chi_\epsilon (z)=\chi\left(\frac{|f(z)|^2}{\epsilon^2}\right)$ where $\{f=0\}$ defines $Z$ as given by hypothesis. Note that for any $v\in\text{Dom}(\dbar)$, Lemma \ref{L:1} implies
\begin{equation}\label{E:2}
\left(U_j,\dbar\left(\chi_\epsilon v\right)\right)=0.
\end{equation}

However, the Cauchy-Schwarz inequality yields
\begin{align*}
\left|\left(U_j,\dbar\left(v-\chi_\epsilon v\right)\right)\right|&=\left|\left(U_j,\dbar v,\chi_\epsilon\dbar v\right)+\left(U_j,\left(\dbar\chi_\epsilon\right)v\right)\right| \\
&\leq \left\| U_j\right\|\, \left\|\dbar v-\chi_\epsilon\dbar v\right\|+\sup\left|\dbar\chi_\epsilon\right|\cdot\text{Vol}\left(\text{supp}(\dbar\chi_\epsilon)\right)\left\|U_j\right\|\, \|v\|.
\end{align*}
The term $\left\|\dbar v-\chi_\epsilon\dbar v\right\|\longrightarrow 0$ as $\epsilon\to 0$ by dominated convergence, since $\dbar v\in L^2$. Additionally,
elementary estimates show
\[
\sup\left|\dbar\chi_\epsilon\right|\cdot\text{Vol}\left(\text{supp}(\dbar\chi_\epsilon)\right)\leq K
\]
for a constant $k$ independent of $\epsilon$. Thus, $U_j\in\text{Dom}\left(\dbar^*\right)$.

The family of equations in \eqref{E:2} show that $\dbar ^*U_j = 0$ on $\Omega _j - Z$. But a simple modification of Lemme 6.9 in \cite{dem-82a} then yields that $\dbar ^*U_j = 0$ on $\Omega _j$.
\end{proof}

   It follows from Lemma \ref{L:2} that $\Box _0 U_j = 0$, and thus $U_j$ is smooth.  Moreover, by \eqref{bound-for-min-ext} and the minimality of $U_j$ we have 
\[
\int _{\Omega_j}|U_j|^2e^{-\vp} dV_{\omega} \le \frac{C}{\delta} \int _{\delta} \int _{Z} \frac{|\xi|^2_{\omega} e^{-\vp}}{|df_Z|^2e^{-\lambda _Z}} \frac{\omega ^{n-1}}{(n-1)!}.
\] 
Finally, $U_j|_{Z\cap\Omega_j}=\xi$.  Indeed, there is a sequence of smooth extensions $h_k$ of $\xi$ that converge to $U_j$ in $L^2(dV_{\omega}, \vp)$, and by convolving all of these with an approximate identity supported near any point of $Z$ and taking a limit, one can see that $U_j$ is an extension of $\xi$ to $\Omega_j$.

\subsection{Conclusion of the proof of Theorem \ref{ambient-main}}

The final step requires choosing a subsequence $U_{j_{\ell}}$ that converges to a form $U$ on $X$ that extends $\xi$.  To this end, first extend $U_j$ by $0$ to all of $X$.  Evidently the $L^2(dV_{\omega}, \vp)$-norm of the extension is uniformly bounded in $j$, so by Alaoglu's Theorem, there exists a subsequence $\left\{U_{j_\ell}\right\}$ converging weakly to $U$ on $X$.

Note that, in the sense of distributions, $\Box _0 U = 0$, so that indeed $U$ is smooth and $\dbar$-closed.  Moreover,
\[
\int _X|U|^2e^{-\vp} dV_{\omega} \le \frac{C}{\delta} \int _{Z} \frac{|\xi|^2_{\omega} e^{-\vp}}{|df_Z|^2e^{-\lambda _Z}} \frac{\omega ^{n-1}}{(n-1)!}
\]
But since $U_j$ converge to $U$ in the sense of distributions, and for any compact subset $K \subset X$ there exists $j >>0$ such that $K \relcomp \Omega _j$, the same argument used above to prove that $U_j|_{\Omega_j}=\xi$, with the $U_j$ taking the place of the $h_k$, shows that $U|_Z = \xi$.  This completes the proof of Theorem \ref{ambient-main}.
\qed

\section{Proof of Theorem \ref{main-i-coh}}

As previously mentioned, Theorem \ref{main-i-coh} is due to Berndtsson in the case where $X$ (and thus $Z$) is compact.  Our approach to the proof is similar in spirit to that of Berndtsson's, but we must take an additional step to overcome an issue that arises from the non-compactness.

\subsection{The complete hyperbolic geometry of a neighborhood of $Z$}

Locally near $Z$, $X-Z$ has the structure of a ball crossed with a punctured disk.  Following Berndtsson \cite[Lemma 2.2]{bo-forms}, we begin by establishing a lemma that constructs a good cut-off function on such a product.

\begin{lem}\label{hyp-lem}
Let $\ve > 0$ be given.  Then there is a smooth function $\rho_{\ve} : B \times \D \to [0,1]$ such that 
\begin{enumerate}
\item[(a)] $\rho_{\ve} (z,t) = 1$ for $|t|\le \ve$ and $\rho _{\ve} (z,t)=0$ for $2\ve\le |t| \le 1$, and
\item[(b)] there exists a constant $C>0$ such that 
\[
\int _{B\times \D} \ii \di \rho _{\ve} \wedge \dbar \rho _{\ve} \wedge (\ii \di \dbar (|z|^2+|t|^2))^{n-1} \le C\ve.
\]
\end{enumerate}
\end{lem}

\begin{rmk}
This lemma is a consequence of the fact that the punctured disk is complete with respect to its Poincar\'e metric $- dd^c (\log \log |t|^{-2}) $.  Indeed, a metric is complete if and only if one can find a proper function whose gradient is bounded with respect to the metric.  For the punctured disk, an example of such a function is the superharmonic function 
\[
\rho (t) = \log \log |t|^{-2},
\]
which satisfies 
\[
\di \rho (t) = \frac{-dt}{t\log |t|^{-2}} \quad \text{and} \quad \di \dbar \rho (t) = - \frac{dt \wedge d\bar t}{|t|^2(\log |t|^{-2})^2},
\]
and therefore 
\[
\di \rho \wedge \dbar \rho = - \di \dbar \rho.
\]
The latter equation is precisely the statement that the norm of $\di \rho$ with respect to the Poincar\'e metric is a positive constant. 
\red
\end{rmk}

\begin{proof}[Proof of Lemma \ref{hyp-lem}]
Take 
\[
\rho _{\ve} (t) := \chi _{\ve} \circ \rho (t),
\]
where $\chi _{\ve}$ is any function that takes values in $[0,1]$, satisfies 
\[
\chi _{\ve} (x) = \left \{ 
\begin{array}{l@{\quad}r}
0 & 0 \le x \le \ve ^{-2}-1\\
1 & x \ge \ve ^{-2} 
\end{array}
\right .  
\]
and has derivative $\chi_{\ve} '(x) \le 2$.  Then by Fubini's Theorem 
\begin{eqnarray*}
&& \int _{B\times \D} \ii \di \rho _{\ve} \wedge \dbar \rho _{\ve} \wedge (\ii \di \dbar (|z|^2+|t|^2))^{n-1} \\
&=& C' \int _{|t|^2 \le e^{-e^{\ve^{-1}}}}\frac{4 dA(t)}{|t|^2(\log |t|^{-2})^2} = 2 C'e^{-1/\ve} \\
&\le&  C \ve
\end{eqnarray*}
for some $C>0$.  The proof is finished.
\end{proof}

It is possible to choose coordinate charts $\{(U_j,z_j)\}_{j \ge 1}$ on $X$ with the following properties.  
\begin{enumerate}
\item[(i)] $Z \subset \bigcup _{j \ge 1} U_j$ and there exists $N \sim 2^{n-1}$ such that each point of $Z$ lies in $\le N$ of the $U_j$.
\item[(ii)] There is a biholomorphic map $F_j :U_j \to B\times \D$ such that, with $F_j := (z'_j,z^n_j)$, 
\[
F_j ^{-1}(0,0)\in Z ,\quad  F_j^{-1}(\{0\}\times \D) \perp Z \quad \text{ and }\quad U_j \cap Z = \{z_j ^n=0\}.
\]
\item[(iii)] For all $j$ there exists $C_j \ge 1$ such that $F_{j*}\omega \le C_j dd^c |z|^2$ on $F_j(U_j) = B\times \D$.
\end{enumerate}
Indeed, one can simply use K\"ahler coordinates in a sufficiently small neighborhood of each point of $Z$.  Usually the norm in property (iii) may have to be much larger than $1$, since the hypersurface $Z$ may become highly curved.  

Now let $\{\chi _j\}$ be a partition of unity subordinate to the open cover $\{U_j\}$ of some neighborhood of $Z$.  Consider the function
\[
\rho := \sum _j \chi _j F_j ^* \rho _{\ve _j}
\]
for constants $\ve _j$ to be chosen in a moment.  Then \begin{equation}\label{rho=1onZ}
\rho |_{U(Z)} \equiv 1,
\end{equation}
for some neighborhood $U(Z)$ of $Z$, whose closure is contained in another neighborhood $V(Z)$ in which $\rho$ is supported.  Moreover, $V(Z)$ can be made as small as desired in the sense that for all $p \in V(Z)$ the distance ${\rm dist}(p, Z)$ is as small as desired, and going as rapidly as desired to zero as $p$ goes to infinity, i.e., exits every compact subset of $X$.  Define $\rho$ to be identically $0$ outside $V(Z)$, thereby extending $\rho$ to all of $X$.  Then for any smooth function $h$ on $X$ and any positive $\ve > 0$, 
\[
\int _{X}  |\di \rho|^2_{\omega}e^{h}\omega ^n \le  2 \sum _{j =1} ^{\infty} \int _{B \times \D}e^{F_{j*}h} \ii \left (  \di \rho _{\ve _j} \wedge \dbar \rho _{\ve _j} + \rho _{\ve _j}^2 F_{j*}(\di \chi _j \wedge \dbar \chi _j)\right ) \wedge F_{j*} \omega ^{n-1}\le\ve^2,
\]
provided $\ve _j$ is chosen sufficiently small.  

\subsection{Extensions of exact intrinsic forms and the proof of Theorem \ref{main-i-coh}}

Now let the notation be that of Theorem \ref{intrinsic-main} (and therefore Theorem \ref{main-i-coh}).  Let $\alpha$ be a smooth, $\dbar$-exact, $L$-valued $(0,q)$-form on $Z$ such that 
\[
\int _Z \frac{|\alpha|^2_{\omega}e^{-\vp}}{|df_Z|^2e^{-\lambda _Z}}  \omega ^{n-1} <+\infty.
\]
By Theorem \ref{intrinsic-main} there exists a smooth, $L$-valued $(0,q)$-form $\beta$ on $X$ such that 
\[
\iota ^* \beta = \alpha \quad \text{and} \quad \int _X |\beta|^2_{\omega} e^{-\vp} \omega ^n < +\infty.
\]
By Lemma \ref{3-weights} there exists a smooth $L$-valued $(0,q-1)$-form $\theta$ on $X$ such that 
\[
\dbar \theta = \beta.
\]
In particular, since $\theta \in L^2_{\ell oc}(X)$, there exists a smooth function $h$ on $X$, which we may assume is as plurisubharmonic as we need, such that 
\[
\int _{X} |\theta|^2 e^{-(\vp+h)} \omega ^{n} < +\infty.
\]
By Theorem \ref{intrinsic-main} there exists a smooth, $L$-valued $(0,q-1)$-form $\eta$ on $X$ such that 
\[
\iota ^*\eta = \theta \quad \text{and} \quad \int _X |\eta|^2_{\omega} e^{-(\vp+h)} \omega ^n < +\infty.
\]
(We don't need the full force of Theorem \ref{intrinsic-main} here; any smooth form $\eta$ in $L^2_{\ell oc}(X)$ would do, after we increase the function $h$.)  Now let $\tilde \alpha = \dbar (\rho \eta)$.  Then by \eqref{rho=1onZ}, $\iota ^* \tilde \alpha = \alpha$.  Moreover, 
\begin{eqnarray*}
\int _X |\tilde \alpha|^2_{\omega} e^{-\vp} \omega ^n &\le & 2  \int _X \rho ^2 |\beta|^2_{\omega} e^{-\vp} \omega ^n + 2 \int _X |\eta|^2_{\omega} e^{-\vp} \di \rho \wedge \dbar \rho \wedge \omega ^{n-1}\\
&\le& 2  \int _{V(Z)} |\beta|^2_{\omega} e^{-\vp} \omega ^n + 2 \left (\int _X |\eta|^2_{\omega} e^{-(\vp+h)} \omega ^n \right ) \left ( \int _{X} |\di \rho|^2_{\omega} e^h \omega ^n\right ).
\end{eqnarray*}
By choosing a sufficiently small neighborhood $V(Z)$ and sufficiently small function $\rho$,  the extension is made as small as desired.

Let $\eta$ be the form to be extended.  Write $\eta = \theta +\dbar \xi$ with $\theta$ orthogonal to the kernel of $\dbar$ (and therefore of minimal $L^2$-norm).  By Theorem \ref{intrinsic-main} there exists an extension of $\theta$  whose $L^2$-norm is controlled by a universal constant times $\kappa (\eta)$.  By the previous section, $\dbar \xi$ has an extension whose $L^2$-norm is as small as we like.  The proof of Theorem \ref{main-i-coh} is complete.
\qed

\section{Proof of Theorem \ref{main-a-coh}}

Theorem \ref{main-a-coh} reduces to Theorem \ref{main-i-coh} by a certain construction that isolates the `intrinsic part' of the ambient form to be extended.

\subsection{The local picture}
In order to clarify the approach, first consider the local picture, i.e., the case $X=B\times \D$.  Let $(z,w)$ be the coordinates.  Suppose $\xi$ is a section of $\Lambda ^qT^{*0,1}_{B\times \D} \to B$, which can be written
\[
\xi = h(z) + f(z) \wedge d\bar w,
\]
where $h$ is a $(0,q)$-form and $f$ is a $(0,q-1)$-forms, both on $B$.  Since $h(z)$ can be thought of as an intrinsic form on $w=0$ and thus is handled by Theorem \ref{main-i-coh}, it suffices to treat the case $h=0$.  Note that $h=0$ if and only if $\iota ^*\xi = 0$.  

Observe then that, with $\rho _{\ve}$ as in the previous section,  
\[
\dbar ((-1)^{q-1} \rho _{\ve}f(z) \bar w) = \rho _{\ve} f(z) \wedge d\bar w + \bar w \left ( f(z) \wedge \dbar \rho _{\ve} + \rho _{\ve} \dbar f (z) \right ).
\]
Thus the form agrees with $\xi$ on $\{w=0\}$ and has arbitrarily small $L^2$ norm on $B\times \D$.

\subsection{Intrinsic and transverse parts}
Let $\xi$ be an ambient $L$-valued $(0,q)$-form on $Z$.  Then we may write 
\[
\xi _{\pitchfork} := \xi - P^* \iota ^* \xi \quad \text{and} \quad \xi = \xi _{\pitchfork}  + P^*\iota ^* \xi.
\]
We remind the reader that $P^*$ is the operation on intrinsic forms defined by the orthogonal projection $P$ with respect to the K\"ahler form via the formula \eqref{kop}.  Observe that since $\iota ^* P^*= {\rm Identity}$, 
\[
\iota ^* \xi _{\pitchfork} = 0.
\]
We introduce the following definition.
\begin{defn}
An ambient $L$-valued $(0,q)$-form $\xi$ is said to be \begin{enumerate}
\item[(i)] {\it transverse} if $\iota ^* \xi = 0$, and 
\item[(ii)] {\it intrinsic} if $\xi = P^* \iota ^* \xi$.
\end{enumerate}
In the decomposition $\xi = \xi _{\pitchfork} + P^*\iota ^*\xi$, the ambient forms $\xi _{\pitchfork}$ and $P^*\iota ^*\xi$ are respectively called the transverse and intrinsic parts of $\xi$.
\red
\end{defn}

\subsection{End of the proof of Theorem \ref{main-a-coh}}

In view of Theorem \ref{main-i-coh}, Theorem \ref{main-a-coh} is proved as soon as the following proposition is established.

\begin{prop}\label{trans-part}
Let $\xi$ be a transverse ambient $L$-valued $(0,q)$-form on $Z$.  Then for any $\ve > 0$ there exists a $\dbar$-closed, $L$-valued $(0,q)$-form $u$ on $X$ such that 
\[
u|_Z = \xi \quad \text{and} \quad \int _X |u|^2 e^{-\vp} \omega ^n < \ve.
\]
\end{prop}

\begin{proof}
The initial claim is that, on $Z$, the section 
\[
d\overline{f_Z}
\]
of $\overline{E_Z}\tensor T^{*0,1}_X|_Z$ is well-defined.  Indeed, in terms of local expressions and transition functions, $f_Z$ is given by holomorphic functions $f_i\in \co (U_i)$ on open sets $U_i \subset X$ such that 
\[
f_i = g_{ij} f_j.
\]
Differentiation given $df_i= g_{ij} df_j + dg_{ij} f_j$, and restricting to $Z$ makes the second term on the right of this equality disappear.  Taking complex conjugates establishes the claim.

Since $\xi$ is transverse and since $f_Z$ generates the ideal sheaf of $Z$, $\xi$ is of the form 
\[
\xi = \alpha \wedge d \overline{f_Z}
\]
for some ambient $(0,q-1)$-form $\alpha$ with values in the line bundle $L \tensor \overline{E_Z^*} \to Z$.  Extend $\alpha$ to a $(0,q-1)$-form $\tilde \alpha$ with values in the line bundle $L \tensor \overline{E_Z^*} \to X$ in a smooth way.  There is no problem in finding such an extension; we are simply extending a smooth section of a smooth vector bundle.  

Now, the line bundle $L \tensor \overline{E_Z^*} \to X$ is not holomorphic, so it does not make sense to apply $\dbar$ to $\tilde \alpha$.  However, the $(0,q-1)$-form 
\[
\tilde \alpha \cdot \overline{f_Z}
\]
takes values in the line bundle $L$, as does 
\[
\hat \alpha := (-1) ^{q-1} \rho _{\ve} \tilde \alpha \cdot \overline{f_Z}.
\]
Now apply $\dbar$ to obtain 
\[
\dbar \hat \alpha = \rho _{\ve} \tilde \alpha \wedge d\overline{f_Z} + \left (\tilde \alpha \wedge \dbar \rho _{\ve} + \rho _{\ve} \beta \right ) \overline{f_Z}
\]
for some $(0,q)$-form $\beta$ with values in the line bundle $L \tensor \overline{E_Z^*}$.  Defining $u := \dbar \hat \alpha$, it is immediate that $u|_Z = \xi$.  Choosing $\rho _{\ve}$ sufficiently small, it is then also clear that the $L^2$-norm of $u$ can be made as small as desired.  The proof of Proposition \ref{trans-part}, and hence of Theorem \ref{main-a-coh}, is now complete.
\end{proof}

\end{document}